\documentclass{amsart}

\usepackage[utf8]{inputenc}
\usepackage[T1]{fontenc}

\usepackage{amsfonts}
\usepackage{amsmath}
\usepackage{amssymb}
\usepackage{amstext}
\usepackage{amsthm}

\usepackage[all]{xy}

\usepackage[]{hyperref}

\newcommand{\C}{{\mathbb C}}

\newcommand{\R}{{\mathbb R}}

\newcommand{\Z}{{\mathbb Z}}

\newcommand{\PP}{{\mathbb P}}

\newcommand{\mM}{{\mathcal M}}
\newcommand{\jJ}{{\mathcal J}}
\newcommand{\oO}{{\mathcal O}}

\newcommand{\op}{\operatorname}

\DeclareMathOperator{\eval}{ev}
\DeclareMathOperator{\sign}{sign}
\DeclareMathOperator{\virdim}{vir-dim}

\newtheorem{thm}{Theorem}[section]
\newtheorem{lemma}[thm]{Lemma}
\newtheorem{cor}[thm]{Corollary}
\newtheorem{prop}[thm]{Proposition}

\theoremstyle{definition}

\usepackage[all]{xy}

\begin{document}

\author[P.\ Ghiggini]{Paolo Ghiggini}

\address[P.\ Ghiggini]{
  Laboratoire de Mathématiques Jean Leray \\
CNRS and Université de Nantes \\
  BP 92208 \\
  2, Rue de la Houssinière \\
  F-44322 Nantes Cedex 03 \\
  FRANCE}

\email{paolo.ghiggini@univ-nantes.fr}

\author[K.\ Niederkrüger]{Klaus Niederkrüger-Eid}

\address[K.\ Niederkrüger]{
  Institut Camille Jordan\\
  Université Claude Bernard Lyon~1\\
  43 boulevard du 11 novembre 1918\\
  F-69622 Villeurbanne Cedex \\
  FRANCE}

\email{niederkruger@math.univ-lyon1.fr}

\title{On the symplectic fillings of standard real projective spaces}

\begin{abstract}
  We prove, in a geometric way, that the standard contact structure on
  $\R\PP^{2n-1}$ is not Liouville fillable for $n \ge 3$ and odd.  We
  also prove for all $n$ that semipositive fillings of such contact
  structures are always simply connected.  Finally we give yet
  another proof of the Eliashberg--Floer--McDuff theorem on the
  diffeomorphism type of the symplectically aspherical fillings of the
  standard contact structure on $S^{2n-1}$.
\end{abstract}

\thanks{The first author was partially supported by the ANR grant
  ANR-16-CE40-017 ``Quantact'', the Simons Foundation and the Centre
  de Recherches Mathématiques, through the Simons-CRM
  scholar-in-residence program, and the grant KAW 2019.0531 from the
  Knut and Alice Wallenberg Foundation.  The second author was
  partially supported by the ANR grant ANR-16-CE40-017 ``Quantact''.
  The first author is grateful to the Mittag-Leffler Institute for its
  hospitality during the program ``Knots, Strings, Symplectic Geometry
  and Dualities'' and to the department of Mathematics of the
  University of Uppsala for hospitality in the Spring semester of
  2021.
  The authors would also like to thank the referee for having
  carefully read the article.  }

\maketitle

\section{Introduction}

The standard contact structure $\xi$ on $S^{2n-1}$ is described in coordinates by the 
equation 
$$\xi = \ker \sum_{j=1}^n (x_j\,dy_j- y_j\,dx_j) .$$
Geometrically, $\xi_p$ is the unique complex hyperplane in $T_pS^{2n-1}$ for every $p \in 
S^{2n-1}$.  
The antipodal involution of $S^{2n-1}$ preserves $\xi$, and therefore induces a 
contact structure on $\R \PP^{2n-1}$ which we still denote by $\xi$. The disc bundle of the 
line bundle $\oO_{\PP^{n-1}}(-2)$ on $\C\PP^{n-1}$ is a strong symplectic filling of 
$(\R\PP^{2n-1}, \xi)$. On the other hand, $\R\PP^{2n-1}$ cannot be the boundary of a 
$2n$-dimensional manifold with the homotopy type of an $n$-dimensional CW complex if $2n-1\ge 5$; see
\cite[Section 6.2]{ekp}. This implies that a real projective space of dimension at least $5$ does not admit any  
Weinstein fillable contact structure. Our main result is the following.

\begin{thm}\label{nonfillability}
  The standard contact structure on $\R\PP^{2n-1}$ admits no
  symplectically aspherical fillings for $n > 1$ and odd.  In particular, it
  is not Liouville fillable.
\end{thm}

These are the first examples of strongly but not Liouville fillable contact structures in high 
dimension. Examples in dimension three were given by the first author in 
\cite{ghiggini:nonstein} using Heegaard Floer homology. In contrast with the high dimensional
situation, the standard contact structure on $\R\PP^3$ is the canonical contact structure on 
the unit cotangent bundle of $S^2$ and therefore is Weinstein fillable.

After a preliminary version of our result (originally for $\R \PP^5$ only) was announced, Zhou 
proved in \cite{zhou:RPn} that $(\R\PP^{2n-1}, \xi)$ is not Liouville fillable if $n \ne 2^k$. He 
also proves similar nonfillability results for some other links of cyclic quotient singularities. 
Zhou's proof uses advanced properties of symplectic cohomology; in contrast our proof is more
 direct, as it relies on the analysis of how a certain moduli space of holomorphic spheres can break, in 
the spirit of McDuff's classification of symplectic fillings of $\R\PP^3$ in 
\cite{mcduff:rationalruled}. 

The strategy is the following. The standard contact structure $\xi$ on $\R \PP^{2n-1}$ admits 
a contact form whose Reeb orbits are the fibres of the Hopf fibration $\R \PP^{2n-1} \to 
\C\PP^{n-1}$. If $(W, \omega)$ is a strong symplectic filling of $(\R\PP^{2n-1}, \xi)$, by a 
symplectic reduction of $\partial W$ (informally speaking, by replacing $\partial W$
 with its quotient by the Reeb flow) we obtain a closed symplectic manifold $(\overline{W}, 
\overline{\omega})$ with a codimension two symplectic submanifold $W_\infty \cong 
\C\PP^{n-1}$ (corresponding to the quotient of $\partial W$) such that $\overline{W} 
\setminus W_\infty$ is symplectomorphic to $\operatorname{int}(W)$; that is, 
$\overline{\omega}|_{\overline{W} \setminus W_\infty} = \omega|_{\operatorname{int}(W)}$. 
The normal bundle of $W_\infty$ is isomorphic to $\oO_{\PP^{n-1}}(2)$. 

We fix a point and a hyperplane in $W_\infty$, and we consider the
moduli space of holomorphic spheres in $\overline{W}$ which are
homotopic to a projective line and pass both through the point and the
hyperplane.  We prove by topological considerations that if
the compactification of that moduli space contains only nodal curves
with at most two irreducible components each of which intersect
$W_\infty$ nontrivially, then some of these nodal curves will
  be composed of two spheres that represent  identical
  homology classes up to torsion.  This implies in particular that the homology
  class of a projective line in $W_\infty$ is  the double of some homology classes in $\overline{W}$ up to torsion.

If $n$ is odd this is a contradiction because the first Chern class of
a line is $n+2$; only at this step we use the hypothesis on the parity
of $n$. This implies that there is either a nodal holomorphic sphere
in $\overline W$ in the homology class of a line of $W_\infty$ with at
least three connected components or a nodal holomorphic sphere with an
irreducible component which is disjoint from $W_\infty$. Since a nodal
sphere intersects $W_\infty$ in exactly two points, in either case at
least one irreducible component must lie entirely in $\op{int}(W)$,
which therefore is not symplectically aspherical.

If $(W, \omega)$ is not symplectically aspherical we lose control on the compactification of 
the moduli space, which is not surprising, given that $(\R\PP^{2n-1}, \xi)$ does admit spherical 
fillings. However, if $W$ is semipositive (and maybe even more generally, using some 
abstract perturbation scheme) we still have enough control to be able to draw conclusions 
about the fundamental group of $W$.
\begin{thm}\label{W is simply connected}
If $(W, \omega)$ is a semipositive symplectic filling of $(\R\PP^{2n-1}, \xi)$, then $W$ is 
simply connected.
\end{thm}
If we apply the same techniques to a symplectically aspherical filling of the standard contact 
structure on $S^{2n-1}$ we obtain that the filling must be diffeomorphic to the ball, a result 
originally due to Eliashberg, Floer and McDuff. 
This is, at least, the fifth proof, after the original one in \cite{mcduff:liouville}, a very similar
one in \cite{oancea-viterbo}, the one in \cite{gnw} using moduli spaces of holomorphic discs
with boundary on a family of LOB's, and the one in \cite{bgz}.
The proof given here is close to the original one, but uses a
different compactification of the filling and is slightly simpler.

\section{The moduli space of lines}

\subsection{The smooth stratum}

By the Weinstein neighbourhood theorem, $W_\infty$ has a tubular neighbourhood that is
symplectomorphic to a neighbourhood of the zero section in the total space of 
$\oO_{\PP^{n-1}}(2)$. Let $\jJ$ be the space of almost complex structures on
$\overline{W}$ which are compatible with $\overline{\omega}$ and coincide with the
natural (integrable) complex structure on $\oO_{\PP^{n-1}}(2)$ on a fixed
neighbourhood of $W_{\infty}$.

For any almost complex structure $J \in \jJ$, any line
$\ell \subset \C\PP^{n-1} \cong W_\infty$ is a $J$-holomorphic
sphere.  Moreover,
\begin{equation}\label{splitting}
  T\overline{W}|_{\ell} \cong \oO_{\PP^1}(2) \oplus \underbrace{ 
\oO_{\PP^1}(1) \oplus \dotsm \oplus \oO_{\PP^1}(1)}_{n-2} 
\oplus \oO_{\PP^1}(2)
\end{equation}
as \emph{holomorphic} vector bundle, where the first $\oO_{\PP^1}(2)$
summand is the tangent bundle of $\ell$, the $(n-2)$-many
$\oO_{\PP^1}(1)$-summands correspond to the normal bundle of $\ell$ in
$\C\PP^{n-1} \cong W_\infty$ and the last  $\oO_{\PP^1}(2)$-summand
is the normal bundle of $W_\infty$ in $\overline{W}$ restricted to $\ell$.

We fix a point $p_0 \in W_\infty$ and a hyperplane
$H_\infty \cong \C\PP^{n-2} \subset W_\infty$ such that
$p_0 \notin H_\infty$.  We denote the moduli space of unparametrised
$J$-holomorphic spheres in $\overline{W}$ that are homotopic
to the lines in $W_\infty \cong \C\PP^{n-1}$ with pointwise
constraints at $p_0$ and $H_\infty$ by $\mM(p_0, H_\infty)$.

We also consider the moduli space $\mM_z (p_0, H_\infty)$ of
unparametrised $J$-holomorphic spheres as above with an extra free
marked point~$z$.  There is a projection
\begin{equation*}
  \mathfrak{f} \colon \mM_z(p_0, H_\infty) \to \mM(p_0, H_\infty).
\end{equation*}
that forgets the marked point.

\begin{lemma}\label{vir-dim}
  $\mM(p_0, H_\infty)$ has expected dimension $2n-2$ and
  $\mM_z(p_0, H_\infty)$ has expected dimension $2n$, where
  $\dim \overline{W}=2n$.
\end{lemma}
\begin{proof}
  The decomposition~\eqref{splitting} gives
  $\langle c_1(T\overline{W}), [\ell] \rangle= n+2$.  The expected
  dimension of $\mM(p_0, H_\infty)$ is
  $$
  \virdim \mM(p_0, H_\infty)= 2 \langle c_1(T\overline{W}), [\ell] \rangle
  + 2n + 4 - 2n -4 -6= 2n-2 \;.
  $$
  The first two terms compute the index of the linearised
  Cauchy-Riemann operator, the third is the contribution of two extra
  marked points, the fourth and the fifth come from the condition that
  the marked points be mapped to $p_0$ and $H_\infty$, and the last is
  the dimension of the biholomorphism group of the sphere.
\end{proof}

The main reason for keeping the almost complex structure integrable near $W_\infty$ is to 
have positivity of intersection between $W_\infty$ and $J$-holomorphic spheres. This fact 
makes our moduli space particularly well behaved, as the following lemma shows.

\begin{lemma}\label{why things work}
  All $J$-holomorphic spheres of $\mM(p_0, H_\infty)$ are simply
  covered and are either lines in $W_\infty$ or intersect $W_\infty$
  transversely in exactly two points.
\end{lemma}
\begin{proof}
Since the algebraic intersection between $W_\infty$ and $\ell$ is $2$, positivity of intersection 
implies that a sphere of  $\mM(p_0, H_\infty)$ is either contained in $W_\infty$, in 
which case it is a line and therefore simply covered, or it intersect $W_\infty$ with total 
multiplicity two. Since the constraints force two distinct intersection points, positivity of 
intersection implies that they are the only ones and that they each have multiplicity one.
\end{proof}
\begin{prop}\label{genericity}
For a generic almost complex structure $J \in \jJ$ the moduli spaces 
$\mM(p_0, H_\infty)$ and $\mM_z(p_0, H_\infty)$ are smooth manifolds of 
dimension $2n-2$ and $2n$ respectively.
\end{prop}
\begin{proof}
  The $J$-holomorphic spheres of $\mM(p_0, H_\infty) $ which are contained in the neighbourhood
  of $W_\infty$ where $J$ is integrable correspond to holomorphic sections of
$\oO_{\PP^{n-1}}$ and therefore admit a decomposition of the restriction of $T\overline{W}$
as in Equation~\eqref{splitting}.  Since the decomposition is into positive holomorphic line
bundles, those spheres are Fredholm regular for every almost complex structures $J \in \jJ$
because the Cauchy-Riemann operator on a positive line bundle over $\C\PP^1$ is surjective
by Serre duality; see \cite[Lemma~3.3.1]{MS}

All other $J$-holomorphic spheres of  $\mM(p_0, H_\infty)$ are Fredholm
regular for a generic $J \in \jJ$ because they are simply covered and intersect the region
where $J$ is generic.  Moreover, the pointwise constraints 
cut out $\mM(p_0, H_\infty)$ transversely for a generic $J$: for spheres near $W_\infty$
this is an explicit computation, and for all other spheres of $\mM(p_0, H_\infty)$ it follows
from \cite[Theorem 3.4.1]{MS} and \cite[Remark 3.4.8]{MS}.  Therefore ${\mathcal 
  M}(p_0, H_\infty)$ is a smooth manifold of the dimension predicted by
Lemma~\ref{vir-dim}.  The corresponding statements for $\mM_z(p_0, H_\infty)$ follow from
those for $\mM(p_0, H_\infty)$.
\end{proof}

\subsection{The compactified moduli space}

Let $\overline\mM(p_0, H_\infty)$ and $\overline\mM_z(p_0, H_\infty)$ be 
the Gromov compactifications of $\mM(p_0, H_\infty)$ and $\mM_z(p_0, 
H_\infty)$ respectively, and let 
$$\overline{\mathfrak{f}} \colon \overline\mM_z(p_0, H_\infty) \to  
\overline\mM(p_0, H_\infty)$$ 
be the forgetful map. We denote $\mM^{\op{red}}(p_0, H_\infty)= 
\overline\mM(p_0, H_\infty) \setminus \mM(p_0, H_\infty)$ and 
$\mM^{\op{red}}_z(p_0, H_\infty)= \overline\mM_z(p_0, H_\infty) \setminus 
\mM_z(p_0, H_\infty)$.

\begin{lemma}\label{bubbles are nice}
If $\overline{W} \setminus W_\infty$ is symplectically aspherical, then every nodal sphere of 
$\mM^{\op{red}}(p_0, H_\infty)$ has exactly two irreducible components, one of which 
intersects $W_\infty$ only at $p_0$ and the other one which intersects $W_\infty$ only at a 
point of $H_\infty$. Both components are simply covered and their intersection with $W_\infty$ 
is transverse.
\end{lemma}
\begin{proof}
  None of the irreducible components of nodal spheres in
    $\overline\mM(p_0, H_\infty)$ is contained in $W_\infty$ because
  any bubble component needs to have positive symplectic area
  strictly smaller than the symplectic area of $\ell$, but the
  homology class of $\ell$ has the smallest positive symplectic area
  in $\C\PP^{n-1}$.  Then by positivity of intersection with
  $W_\infty$ a nodal sphere must intersect $W_\infty$ in at most two
  points.  Moreover, if $\overline{W} \setminus W_\infty$ is
  symplectically aspherical, every irreducible component must
  intersect $W_\infty$.  This implies that there are exactly two
  irreducible components and the intersection of each with $W_\infty$
  has multiplicity one.  Therefore both components are simply covered.
\end{proof} 

This lemma implies that we have enough topological control on the nodal curves to show that 
they have  smooth moduli spaces.

\begin{lemma} \label{Mred is smooth}
The moduli space $\mM^{\op{red}}(p_0, H_\infty)$ is a smooth manifold of dimension
$2n-4$. The forgetful map 
$$\mathfrak{f}^{\op{red}} \colon \mM_z^{\op{red}}(p_0, 
H_\infty) \to \mM^{\op{red}}(p_0, H_\infty)$$
is a locally trivial fibration with fibre $S^2 \vee S^2$.
\end{lemma}
\begin{proof}
By Lemma~\ref{bubbles are nice} the irreducible components of the nodal spheres of 
the moduli space $\mM^{\op{red}}(p_0, H_\infty)$ are simply covered and intersect 
the region where the almost complex structure can be made generic. Then the statement 
follows from \cite[Theorem 6.2.6]{MS}.
\end{proof}

What we have shown so far about $\overline\mM_z(p_0, H_\infty)$ is enough to show that the
image of the evaluation is a pseudocycle in $\overline{W}$ (see \cite[Section 6.5]{MS}).
While pseudocycles are good enough for certain degree arguments, as for example in
Section~\ref{sec: pi_1}, in the proof of our main theorem we will need a differentiable
structure on the compactified moduli spaces. 
The reason is that our proof is based on studying the properties
of the homology class $[\eval^{-1}(\ell)]$ in $\overline\mM_z(p_0, H_\infty)$, and we
have found no better way to show that $\eval^{-1}(\ell)$ is well-behaved
than by the implicit function theorem. Since there is no reason to expect that $\ell$ can be
made disjoint from the image of  $\mM_z^{\op{red}}(p_0, H_\infty)$ by the evaluation map,
we need a smooth structure on the whole compactified moduli space, and not only in its
irreducible part. Luckily our moduli space is simple enough that standard gluing theory (as
explained, for example, in \cite{MS}) already produces a smooth structure.

In the rest of the
section we will sketch the construction of a $C^1$-structure, which in turn can be promoted
to a smooth structure by a classical result in differential topology;  see for example \cite[Section~2.2, Theorem~2.10]{hirsch:diff_top}.

In order to endow $\overline\mM(p_0, H_\infty)$ and $\overline\mM_z(p_0, H_\infty)$ with the
structure of a $C^1$-manifold, we exhibit them as union of two open manifolds which are
patched together by a diffeomorphism of class $C^1$. The two patches are the irreducible
strata $\mM(p_0, H_\infty)$ and $\mM_z(p_0, H_\infty)$ on one side, and suitable fibrations
over the reducible strata $\mM^{\op{red}}(p_0, H_\infty)$ and
$\mM^{\op{red}}_z(p_0, H_\infty)$ on the other hand.  The fibres
  of those fibrations are, roughly speaking, spaces of gluing parameters.  We are going to
  define them momentarily.

The first step in our construction is to introduce local gauge fixing conditions to simplify the
presentation of the moduli spaces, so that they almost become spaces of parametrised
$J$-holomorphic spheres.
Identify the neighbourhood of $W_\infty$ in $\overline{W}$ with a
neighbourhood of the $0$-section of $\oO_{\PP^{n-1}}(2)$ as
already discussed above.  The hyperplane~$H_\infty$ is the $0$-set of
a holomorphic section~$\sigma$ in $\oO_{\PP^{n-1}}(1)$, and
it follows that $\sigma^2$ is a section of
$\oO_{\PP^{n-1}}(2)$ that has a zero of order two along
$H_\infty$.  Multiplying $\sigma^2$ with a small constant, we can
assume that its image lies in an arbitrarily small neighbourhood of the
$0$-section.  Its graph is a $J$-holomorphic hypersurface  in
$\overline{W}$ that we will call $\widetilde{W}_\infty$.
In particular 
$\widetilde{W}_\infty \cap W_\infty=H_\infty$ and $T\widetilde{W}_\infty|_{H_\infty}= 
TW_\infty|_{H_\infty}$. Then every sphere of $\mM(p_0, H_\infty)$ which is not 
contained in $W_\infty$ intersects $\widetilde{W}_\infty$ in two points: one in $H_\infty$ and 
one in $\widetilde{W}_\infty \setminus H_\infty$.

Let $\widetilde\mM(p_0, H_\infty)$ be the open subset of 
$\mM(p_0, H_\infty)$ consisting of those spheres which are not contained in 
$W_\infty$  and $\widetilde\mM_z(p_0, H_\infty)$ the corresponding open
subset of $\mM_z(p_0, H_\infty)$.  By the discussion in the previous paragraph,
we can fix a parametrisation for every element in $\widetilde\mM(p_0, H_\infty)$ identifying
this moduli space with the set of $J$-holomorphic maps $u \colon S^2 \to
\overline{W}$ whose image is homotopic to $\ell$ but not contained in $W_\infty$, and such 
that  $u(0)=p_0$, $u(1) \in \widetilde{W}_\infty \setminus H_\infty$ and $u(\infty) \in H_\infty$.

We also denote by $\widetilde\mM^{\op{red}}(p_0, H_\infty)$ the set of pairs of
$J$-holomorphic maps $(u^0, u^\infty)$, with $u^0, u^\infty \colon S^2 \to \overline{W}$,
such that
\begin{align*}
  u^0(0) &= p_0 \;, & u^0(\infty) &= u^\infty(0) \;, &  u^\infty(\infty) &\in H_\infty \;,  \\
  u^0(1) &\in \widetilde{W}_\infty \setminus H_\infty \;, & |d_\infty u^\infty| &= 1 &
\end{align*}
and the image of the ``connected sum map'' $u^0 \# u^\infty \colon S^2 \# S^2 \cong S^2 \to \overline{W}$ is homotopic to $\ell$.    Here $|d_\infty u^\infty|$ denotes the
norm of the differential of $u^\infty$ at $\infty \in S^2$ computed with respect to the round
metric on $S^2$ and the metric induced by $J$ and $\overline{\omega}$ on $\overline{W}$.
Only the component~$u^0$ meets
$\widetilde{W}_\infty \setminus H_\infty$, because $u^1$ already
intersects $\widetilde{W}_\infty$ in $H_\infty$.

The group of complex numbers of modulus $1$ acts on 
$\widetilde\mM^{\op{red}}(p_0, H_\infty)$ by
$$\theta \cdot (u^0, u^\infty)= \bigl(u^0, u^\infty(\theta^{-1}
\cdot)\bigr)$$
and the quotient by this action is $\mM^{\op{red}}(p_0, H_\infty)$.  This
implies that the projection $\widetilde\mM^{\op{red}}(p_0, H_\infty) \to
\mM^{\op{red}}(p_0, H_\infty)$ is a principal $S^1$-bundle which need not be trivial.

\medskip

The second step is to define the fibrations over $\mM^{\op{red}}(p_0, H_\infty)$ and $\mM^{\op{red}}_z(p_0, H_\infty)$ which give one of the two patches.
Let $\pi \colon S^2 \times S^2 \dashrightarrow S^2$ be the rational map $\pi(x,y)= y/x$, 
which is not defined at the points $(0,0)$ and $(\infty, \infty)$. If we make $S^1$ act on 
$S^2 \times S^2$ by $\theta \cdot (x,y) = (x, \theta y)$ and on $S^2$ by $\theta 
\cdot w = \theta w$, then $\pi$ is $S^1$-equivariant. Let $\mathfrak{X}$ be the
smooth variety obtained by blowing up $S^2 \times S^2$ at $(0,0)$ and $(\infty, \infty)$.
The action of $S^1$ on $S^2 \times S^2$ induces an action on $\mathfrak{X}$,  and $\pi$
extends to a smooth $S^1$-equivariant map $\pi \colon \mathfrak{X} \to S^2$. 

We denote by $D_\epsilon$ the disc with centre in $0$ and radius $\epsilon$ in $\C \subset 
S^2$ and $\mathfrak{X}_\epsilon = \pi^{-1}(D_\epsilon)$.  We define $E_\epsilon = \widetilde\mM^{\op{red}}(p_0, H_\infty) \times_{S^1} D_\epsilon$ and $X_\epsilon = 
 \widetilde\mM^{\op{red}}(p_0, H_\infty)  \times_{S^1} \mathfrak{X}_\epsilon$.  
Both $E_\epsilon$ and $X_\epsilon$ are fibre bundles over  $\mM^{\op{red}}(p_0, H_\infty)$ and there is a bundle map
$$\xymatrix{
X_\epsilon \ar[rr]^{\varpi} \ar[dr] & & E_\epsilon \ar[dl] \\
& \mM^{\op{red}}(p_0, H_\infty) &
} $$

The zero section $E_0$ of $E$ is, of course, diffeomorphic to $\mM^{\op{red}}(p_0, H_\infty)$. 
Let us denote $\mathfrak{X}_0= \pi^{-1}(0)$ and $X_0= \varpi^{-1}(E_0)$.
We observe that $X_0$ is diffeomorphic to  $\mM^{\op{red}}_z(p_0, H_\infty)$.  In fact  $X_0 = \widetilde{\mM}^{\op{red}}(p_0, H_\infty) \times_{S^1} \mathfrak{X}_0$ and $\mathfrak{X}_0 = \{ y=0 \} \cup \{ x= \infty \}$, so we can identify the sphere  $\{ y=0 \}$ with the domain of $u^0$  and the sphere $ \{ x= \infty \}$ with the domain of $u^\infty$.

\medskip

The third, and last, step is the definition of the gluing maps between the
two patches.  Let $\dot E_\epsilon$ denote $E_\epsilon$ with the zero section
removed, and $\dot X_\epsilon$ the preimage of $\dot E_\epsilon$. We can identify
$\dot E_\epsilon \cong [\epsilon^{-1}, +\infty) \times  \widetilde\mM^{\op{red}}(p_0, H_\infty)$, 
and therefore standard gluing theory (see for example \cite[Chapter~10]{MS}) yields
 a $C^1$-embeddings $\mathfrak{g} \colon \dot E_\epsilon \to
  \widetilde\mM (p_0, H_\infty)$: if $(r, (u^0, u^{\infty})) \in
[\epsilon^{-1}, +\infty) \times  \widetilde\mM^{\op{red}}(p_0, H_\infty) \cong \dot E_\epsilon$,
then $\mathfrak{g}(r, (u^0, u^{\infty})) \in \widetilde\mM(p_0, H_\infty)$ is the
$J$-holomorphic sphere obtained by gluing $u^0$ and $u^{\infty}$ with gluing parameter
$R=\sqrt{r}$; see \cite[Section~10.1]{MS}.

We can also define a $C^1$-embedding $\mathfrak{G} \colon \dot X_\epsilon \to
\widetilde\mM_z (p_0, H_\infty)$ such  that the
diagram
$$\xymatrix{
  \dot X_\epsilon \ar[d]_{\mathfrak{G}} \ar[r]^{\varpi} & \dot E_\epsilon
  \ar[d]_{\mathfrak{g}} \\
  \widetilde\mM_z (p_0, H_\infty) \ar[r]^{\mathfrak{f}} &  \widetilde\mM(p_0, H_\infty)
}$$
commutes.
To define $\mathfrak{G}$ it is enough to identify $\varpi^{-1}(e)$ with the
domain of $\mathfrak{g}(e)$ for all $e \in \dot E_\epsilon$. We recall that $\mathfrak{g}(e)$
is obtained by deforming a preglued map $\mathfrak{p}(e)$ with the same domain, so it is
enough to identify (in a smooth way) $\varpi^{-1}(e)$ with the domain of the preglued map
$\mathfrak{p}(e)$, whose construction we sketch now.
Denote $e=[((u^0, u^\infty),t)]$ with
$t \in D_\epsilon$ and $(u^0, u^\infty) \in \widetilde{\mM}^{\op{red}}(p_0, H_\infty)$, and $S_t= \pi^{-1}(t)$.
We define $\widetilde{\mathfrak{p}}(u^0, u^\infty,t) \colon S_t \to
\overline{W}$ by 
$$\widetilde{\mathfrak{p}}(u^0, u^\infty,t)(x,y)= \begin{cases} u^0(x) & \text{if } |x| <
  \frac{1}{2 \sqrt{|t|}}, \\ u^{\infty}(y) & \text{if } |x|> \frac{2}{\sqrt{|t|}} \end{cases}$$
and in the region $\left \{ (x,y) \in S_t :  \frac{1}{2 \sqrt{|t|}} \le |x| \le \frac{2}{\sqrt{|t|}} \right \}$ we interpolate between $u^0$ and $u^\infty$ while remaining close to $u^0(\infty)= u^\infty(0)$. It is possible to chose the interpolation compatibly with the
$S^1$-actions on $\widetilde{\mM}^{\op{red}}(p_0, H_\infty)$ and $\mathfrak{X}_\epsilon$
so that the map $\widetilde{\mathfrak{p}}(u^0, u^\infty,t)$ induces a well defined map
$\mathfrak{p}(e) \colon \varpi^{-1}(e) \to \overline{W}$.
If we choose a representative
$(u^0, u^\infty,t)$ of $e$ where $t$ is real positive and we define $R=\frac{1}{\sqrt{t}}$,
we see that the pregluing $\mathfrak{p}(e)$ is the same as the pregluing defined in
\cite[Section~10:2]{MS}, up to a holomorphic change of coordinates and the introduction
of a  constant $\delta$, which is necessary for the gluing estimates, but does not change in
any significant way the geometric picture we have described.

Combining \cite[Theorem~6.2.6]{MS} with the discussion above we obtain
the following structural result for the moduli spaces we are
interested in.

\begin{prop}
  The moduli spaces $\overline\mM(p_0, H_\infty)$ and
  $\overline\mM_z(p_0, H_\infty)$ are closed and orientable
  $C^1$-manifolds and there is a $C^1$-map
  $$\overline{\mathfrak{f}} \colon  \overline\mM_z(p_0, H_\infty) \to 
  \overline\mM(p_0, H_\infty)$$
  which forgets the marked point.
\end{prop}

While $\overline\mM(p_0, H_\infty)$ is not a priori connected, since we have not 
ruled out that a $J$-holomorphic sphere could be homotopic to a line $\ell \subset 
W_\infty$ but not homotopic through $J$-holomorphic spheres, we can assume without loss of 
generality that $\overline\mM(p_0, H_\infty)$ is connected  by restricting our attention 
to the connected component which contains a line $\ell \subset W_\infty$.

\section{Proof of the main theorem}

\subsection{Degree of the evaluation map}

Let $\eval \colon \overline\mM_z(p_0, H_\infty) \to \overline{W}$ be the 
evaluation map at the free marked point. 

\begin{lemma} \label{U exists}
There is an open subset $U \subset \overline{W}$ such that every $J$-holomorphic sphere of
$\overline\mM(p_0, H_\infty)$ passing through a point of $U$ belongs to ${\mathcal 
M}(p_0, H_\infty)$ and its image is contained in the neighbourhood
of $W_\infty$ on which $J$ is integrable.
\end{lemma}
\begin{proof}
 Choose a point $q_0 \in W_\infty \setminus H_\infty$ such that $q_0 \ne p_0$. The unique 
line $\ell_0$ in $W_\infty$ passing through $p_0$ and $q_0$ also intersects $H_\infty$, and 
therefore determines an element of $\mM(p_0, H_\infty)$. Moreover, any
sphere of $\mM(p_0, H_\infty)$ passing through $q_0$ intersects $W_\infty$ in three 
points, and therefore must be contained in it, so it is equal to $\ell_0$.

Since none of the nodal spheres passes through $q_0$, and since $\mM_z^{\op{red}}(p_0, H_\infty)$ 
is compact, there is a neighbourhood $U$ of $q_0$ in $\overline{W}$ such that 
$\eval^{-1}(U) \subset \mM_z(p_0, H_\infty)$. 

After possibly reducing the size of $U$, we can assume that every $J$-holomorphic sphere of 
$\mM(p_0, H_\infty)$ passing through $U$ is contained in the neighbourhood of 
$W_\infty$ on which $J$ is integrable. Suppose on the contrary that there is a sequence $[u_n]$ 
of elements of  $\mM(p_0, H_\infty)$ and a sequence of points $q_n \in \overline{W}$ 
converging to $q_0$ such that the image of $u_n$ contains $q_n$, but is not contained in 
some fixed neighbourhood of $W_\infty$. Then by Gromov compactness there is a 
subsequence of $[u_n]$ converging to a (possibly nodal) $J$-holomorphic sphere of 
$\overline\mM(p_0, H_\infty)$  passing through $q_0$ and not contained in the fixed 
neighbourhood of $W_\infty$. This is a contradiction because the only element of 
$\overline\mM(p_0, H_\infty)$ passing through $q_0$ is $\ell_0$, which is contained in 
$W_\infty$.
\end{proof}

\begin{lemma}\label{evaluation has degree one}
  The evaluation map
  $\eval \colon \overline\mM_z(p_0, H_\infty) \to
  \overline{W}$ has degree one.
\end{lemma}
\begin{proof}
  Let $U$ be the neighbourhood defined in Lemma~\ref{U exists}.  We
  will show that $\# \eval^{-1}(q)=1$ for every $q \in U$.

  Since all $J$-holomorphic 
spheres passing through $U$ are contained in the neighbourhood where $J$ is integrable, we 
can pretend we are working in the total space of $\oO_{\PP^{n-1}}(2)$. Given 
$q \in U$, let $\overline{q}$ be its projection to $\C\PP^{n-1}\cong W_\infty$. Any 
$J$-holomorphic sphere of $\mM(p_0, H_\infty)$ passing through $q$ projects to the 
unique line $\ell_q$ in $W_\infty$ passing through $p_0$ and $\overline{q}$.
The sphere itself  corresponds then 
to a section of $\oO_{\PP^{n-1}}(2)|_{\ell_q} \cong \oO_{\PP^1}(2)$ which 
vanishes at $p_0$ and at $p_\infty = \ell_q \cap H_\infty$. The space of sections of 
$\oO_{\PP^1}(2)$  vanishing at $p_0$ and $p_\infty$ has complex dimension one, 
and thus there is a unique such section for any point~$q$ in the fibre of 
$\oO_{\PP^1}(2)$ over $\overline{q}$.

This shows that $\# \eval^{-1}(q)=1$ for every $q \in U$, and since
$U$ is open, by Sard's theorem it contains a regular value of the
evaluation map.  This proves that $\eval$ has degree one.
\end{proof} 

It is important to have a degree one map because such maps induce
surjections in homology.  More generally, we have the following lemma.

\begin{lemma}\label{why we like it}
  Let $f \colon X \to Y$ be a smooth map between closed oriented
  smooth manifolds of the same dimension.  Assume that $f$ has degree
  $d$, and let $S \subset Y$ be a compact, oriented $k$-dimensional
  submanifold that is transverse to $f$.

  Then it follows that $S' := f^{-1}(S)$ has an induced orientation
  and, with that orientation, we have the equality
  $$
  f_*\bigl([S']\bigr) =  d\, [S]
  $$
  in $H_k(Y; \Z)$.
\end{lemma}
\begin{proof}
A submanifold $S$ is transverse to a map $f$ if, for every $y \in S$ and $x \in f^{-1}(y)$ 
we have $T_yS \oplus d_xf(T_xX)=T_yY$. This property implies that 
\begin{itemize}
\item $S'= f^{-1}(S)$ is a compact submanifold of $X$, and
\item $df$ defines an isomorphism between the normal bundle of 
$S'$ and the normal bundle of $S$. 
\end{itemize}
The orientations of $S$ and $Y$ determine an orientation
of the normal bundle of $S$.  This in turn induces an orientation of the normal bundle 
of $S'$ via $df$ which, combined with the orientation of $X$ induces the orientation
of $S'$.

Let $f_S \colon S' \to S$ be the restriction of $f$. The condition on the normal bundles implies 
that the regular values of $f_S$ are also regular values of $f$. If $y$ is a regular value of 
$f_S$, then
\begin{align*}
  \deg(f_S)= & \sum \limits_{x \in f_S^{-1}(y)} \sign(d_xf_S) \\
  \deg(f)= & \sum \limits_{x \in f^{-1}(y)} \sign(d_xf) .
\end{align*} 

Since $f_S^{-1}(y)=f^{-1}(y)$ by the definition of $f_S$ and
$\sign(d_xf_S) = \sign(d_xf)$ because $df$ is an orientation
preserving isomorphism between the normal bundles, we obtain
$\deg(f_S) = \deg(f) = d$.

Now we consider the commutative diagram
$$\xymatrix{
H_k(S'; \Z) \ar[d] \ar[r]^{(f_S)_*} & H_k(S; \Z) \ar[d] \\
H_k(X; \Z) \ar[r]^{f_*} & H_k(Y; \Z)
}$$ 
where the vertical arrows are induced by the inclusions. The fundamental class of $S'$ is 
mapped  by $(f_S)_*$ to $\deg (f_S)$ times the fundamental class of $S$. The homology classes
 $[S']$ and $[S]$ are the images of the fundamental classes of $S'$ and $S$ in $H_k(X; \Z)$ 
and $H_k(Y; \Z)$ respectively, and therefore $f_*[S']= \deg(f_S)\, [S]=  d\, [S]$.
\end{proof}

\subsection{Decomposition of the line}

The following lemma is a warm up which illustrates how to derive topological implications 
from  Lemma~\ref{evaluation has degree one}. 

\begin{lemma}\label{noncompact}
The moduli space $\mM_z(p_0, H_\infty)$ is not compact.
\end{lemma}
\begin{proof}
The moduli space $\mM_z(p_0, H_\infty)$ is an $S^2$-bundle over  $\mM(p_0,
 H_\infty)$ with two distinguished sections $\eval^{-1}(p_0)$ and $\eval^{-1}(H_\infty)$. 
Then $\mM_z(p_0, H_\infty) \setminus \eval^{-1}(H_\infty)$ retracts onto
$\eval^{-1}(p_0)$. This implies that 
$$\eval_* \colon H_k\bigl(\mM_z(p_0, H_\infty) \setminus \eval^{-1}(H_\infty); \Z\bigr) 
\to H_k(\overline{W}; \Z)$$
is trivial whenever $k >0$.

Let $\ell \subset \overline{W}$ be an embedded sphere  which is
homologous to a line in $W_\infty$ but disjoint from $H_\infty$.  It
is possible to find such a sphere because $H_\infty$ has
codimension~$4$ in $\overline{W}$, but, in general, $\ell$ will not be
holomorphic.  We perturb $\ell$ to be transverse to the evaluation map
and we denote $\eval^{-1}(\ell)$ by $\ell'$.  If
$\mM_z(p_0, H_\infty)$ is compact, $\eval_*([\ell'])=[\ell]$
by Lemma~\ref{why we like it}.  Since
$\ell \cap H_\infty = \emptyset$, it follows that $\ell'$ does
  not intersect $\eval^{-1}(H_\infty)$.  The previous paragraph
implies then that $[\ell] = \eval_*([\ell']) = 0$.  This is a
contradiction because $\ell$ is homologous to a symplectic sphere, and
therefore is nontrivial in homology.
\end{proof}

Lemma~\ref{noncompact} tells us thus that
$\mM^{\op{red}}(p_0, H_\infty)$ is nonempty.  We decompose it into
connected components
$$\mM^{\op{red}}(p_0, H_\infty)= \mM^{(1)}(p_0, H_\infty) \sqcup \dotsm \sqcup
\mM^{(N)}(p_0, H_\infty)$$
and, correspondingly, we decompose the moduli space with a free marked point into 
connected components
$$\mM^{\op{red}}_z(p_0, H_\infty)= \mM^{(1)}_z(p_0, H_\infty) \sqcup \dotsm \sqcup
\mM^{(N)}_z(p_0, H_\infty).$$ Each $\mM^{(i)}_z(p_0, H_\infty)$ is an
$S^2 \vee S^2$-bundle over $\mM^{(i)}(p_0, H_\infty)$ with three
distinguished sections: one, denoted ${\mathcal S}_0^{(i)}$, where the
free marked point is mapped to $p_0$, one, denoted
${\mathcal S}_\infty^{(i)}$, where the free marked point is mapped to
$H_\infty$, and one, denoted ${\mathcal S}_{\mathfrak{n}}^{(i)}$,
where the free marked point lies on the node.\footnote{Strictly
  speaking ghost bubbles appear in these three cases and we tacitly
  contract them. We ignore this technical complication as it has no
  topological consequence.}  Therefore we can see each
$\mM^{(i)}_z(p_0, H_\infty)$ as the union of two sphere bundles
${\mathcal N}_0^{(i)}$ and ${\mathcal N}_\infty^{(i)}$ over
$\mM^{(i)}(p_0, H_\infty)$ glued together along the section
${\mathcal S}_{\mathfrak{n}}^{(i)}$. An element of
$\mM_z^{(i)}(p_0, H_\infty)$ belongs to ${\mathcal N}_0^{(i)}$ when
the free marked point lies in the domain of the irreducible component
passing through $p_0$, and to ${\mathcal N}_\infty^{(i)}$ when the
free marked point lies in the domain of the irreducible component
passing through $H_\infty$.  We denote the homology classes
  representing the fibres of ${\mathcal N}^{(i)}_0$ and of
  ${\mathcal N}^{(i)}_\infty$ by $A_0^{(i)}$ and by $A_\infty^{(i)}$
  respectively.

\medskip

Our aim is to show that there is a nodal curve in the compactification
of $\mM(p_0, H_\infty)$ that is composed of two holomorphic spheres
representing homology classes which are equal up to torsion.

  A nodal curve in $\mM^{(i)}(p_0, H_\infty)$ is composed of two
  holomorphic spheres that are fibres of
  ${\mathcal N}_*^{(i)}$ for $* = 0$ or $* = \infty$ so that
  $\eval_* \bigl(A_0^{(i)}\bigr) + \eval_* \bigl( A_\infty^{(i)}\bigr)
  = [\ell]$ in $H_2(\overline{W}; \Z)$.

  The pull-back of the symplectic form $\overline{\omega}$ is
  cohomologically nontrivial on the fibres of ${\mathcal N}_*^{(i)}$
  for any $i = 1, \ldots , N$ and $* \in \{0, \infty \}$.  Therefore,
  by the Leray-Hirsch Theorem (see \cite[Theorem 5.11]{Bott-Tu} for
  its cohomological form),
  \begin{equation}\label{Leray-Hirsch}
    H_2({\mathcal N}_*^{(i)}; \Z)
    \cong H_2\bigl({\mathcal S}_*^{(i)}; \Z\bigr) \oplus H_2(S^2; \Z) \cong
    H_2\bigl({\mathcal S}_{\mathfrak{n}}^{(i)}; \Z\bigr) \oplus H_2(S^2; \Z),
  \end{equation}
  where the summand $H_2(S^2; \Z)$ is generated by a fibre of
  ${\mathcal N}^{(i)}_*$.

 In the next two lemmas we use the Leray-Hirsch Theorem to gain homological information
   on the evaluation maps, and on the components of the nodal curves.  Given homology classes~$A$ and $B$ of
 complementary degrees (in the same manifold), we denote by $A \cdot B$ their
 intersection product.  If $A$ and $B$ are  represented by closed, oriented submanifolds
 which intersect transversely, $A \cdot B$ is the algebraic count of intersection points.

  \begin{lemma}\label{Sn is trivial}
    The map
    \begin{equation*}
      (\eval|_{{\mathcal S}_\mathfrak{n}^{(i)}})_* \colon 
      H_2\bigl({\mathcal S}_\mathfrak{n}^{(i)}; \Z\bigr) \to
      H_2(\overline{W}; \Z)
    \end{equation*}
    is trivial for every $i=1, \dotsc, N$.
  \end{lemma}
  \begin{proof}
    By Equation~\eqref{Leray-Hirsch} every class
    $c \in H_2({\mathcal S}_\mathfrak{n}^{(i)}; \Z)$ can be written as
    the sum of a class in $H_2({\mathcal S}^{(i)}_0; \Z)$ and a
    multiple of the class of the fibre. Since ${\mathcal S}_0^{(i)}$
    is mapped to $p_0$, we obtain
    $(\eval|_{{\mathcal S}_0^{(i)}})_*=0$, and thus
    $\eval_*(c)= k\, \eval_*\bigl(A^{(i)}_0\bigr)$.  By
    Lemma~\ref{bubbles are nice}
    $\eval\bigl({\mathcal S}^{(i)}_{\mathfrak{n}}\big) \cap W_\infty =
    \emptyset$ while
    $\eval_*\bigl(A_0^{(i)}\bigr) \cdot [W_\infty] =1$ so that
    \begin{equation*}
      0= \eval_*(c) \cdot [W_\infty]
      = k \, \eval_*\bigl(A_0^{(i)}\bigr) \cdot [W_\infty] =k  \;. \qedhere 
    \end{equation*}
  \end{proof}

 Let $\eval_\infty^{(i)} \colon {\mathcal S}_\infty^{(i)} \to H_{\infty}$ be the
    restriction of $\eval \colon \overline{\mathcal M}_z(p_0, H_\infty) \to \overline{W}$.

    \begin{lemma}\label{ultimo}
      If $\deg(\eval_\infty^{(i)}) \ne 0$, then $[\ell]= 2 \eval_*\bigl([A_\infty^{(i)}]\bigr)$ modulo
      torsion.
    \end{lemma}
\begin{proof}
    Let $\ell$ be now a line that lies in $H_\infty$ and perturb it
  (inside $H_\infty$) to make it transverse to
  $\eval_\infty^{(i)}\colon {\mathcal S}^{(i)}_\infty \to H_\infty$.
  Then $\ell'_i := (\eval_\infty^{(i)})^{-1}(\ell)$ is a smooth
  submanifold of ${\mathcal S}^{(i)}_\infty$ which, by Lemma~\ref{why
    we like it}, satisfies
  $\bigl(\eval_\infty^{(i)}\bigr)_*([\ell'_i]) = \kappa_i\, [\ell]$
  with $\kappa_i:= \deg \bigl(\eval_\infty^{(i)})$.  Using that
  $\eval$ and $\eval_\infty^{(i)}$ commute with the corresponding
  inclusions, we also obtain $\eval_*([\ell'_i]) = \kappa_i\, [\ell]$.

  According to Equation~\eqref{Leray-Hirsch}, we can represent
  $[\ell_i']$ as
  \begin{equation*}
    [\ell_i'] = d\, A_\infty^{(i)} + c
  \end{equation*}
  for some $d \in \Z$ and
  $c \in H_2({\mathcal S}_{\mathfrak{n}}^{(i)}; \Z)$.  This combined
  with Lemma~\ref{Sn is trivial} shows that
  $d\, \eval_*\bigl(A_\infty^{(i)}\bigr) = \kappa_i\, [\ell]$, and by
  intersecting with $W_\infty$ we obtain $d = 2 \kappa_i$ so that
  \begin{equation*}
    \kappa_i \, \bigl(2\,\eval_*\bigl(A_\infty^{(i)}\bigr) - [\ell]\bigr) = 0 \;.
    \qedhere
  \end{equation*}
\end{proof}

The next step is to show that $\deg(\eval_\infty^{(i)}) \ne 0$ for at
least one $i \in \{1, \ldots, N \}$.  This will be the goal of the
next lemmas.

\begin{lemma}\label{remove unnecessary intersection points}
  Let $X$ be a compact oriented $n$-dimensional manifold containing
  two closed oriented submanifolds~$S$ and $Y$.  Suppose that
  $\dim S + \dim Y = n$, that $\dim Y \ge 2$, and that $S$ and $Y$
  intersect transversely.
  
  Then there is an oriented submanifold~$S'$ that is homologous to $S$
  and that intersects $Y$ transversely in exactly $\bigl|[S]\cdot [Y]\bigr|$
  many points.

  More precisely, we can choose an arbitrarily small neighbourhood of
  $Y$ such that $S$ and $S'$ agree outside this neighbourhood.
  Furthermore, given a compact subset $Y' \subset X$ that is disjoint
  from $S$, and that intersects $Y$ in a codimension~$2$ submanifold,
  we can additionally assume that $S'$ is also disjoint from $Y'$.
\end{lemma}
\begin{proof}
  We obtain $S'$ by attaching certain $1$-handles to $S$.
  
  If the number of intersection points of $S$ and $Y$ does not agree
  with $\bigl|[S]\cdot [Y]\bigr|$, then there needs to be a pair of intersection
  points $\{x_-,x_+\} \subset S \cap Y$ of opposite sign.  Choose an
  embedded path~$\gamma$ in $Y$ with end points~$x_-$ and $x_+$ that
  avoids any other intersection point in $S \cap Y$ and also
  $Y'\cap Y$, if such a $Y'$ has been chosen.

  Identify a tubular neighbourhood of $Y$ with the normal bundle of
  $Y$, and assume that $S$ corresponds in this neighbourhood to the
  fibres of the normal bundle over the points in $S \cap Y$.  Note
  that the normal bundle is naturally oriented by the orientations of
  $Y$ and $X$.
  
  The restriction of the disk bundle over $\gamma$ is a solid cylinder
  $D^k \times [0,1]$ such that $D^k \times \{0,1\}$ is a neighbourhood
  of $\{x_-,x_+\}$ in $S$.  The solid cylinder is naturally oriented,
  and the orientation of $S$ at $\{x_-,x_+\}$ is equal to the boundary
  orientation of $D^k \times [0,1]$.

  Remove $D^k \times \{0,1\}$ from $S$, and glue instead the tube
  $(\partial D^k) \times [0,1]$ along the boundary of the two holes
  that we have created in $S$ (abstractly this corresponds to performing
  an index~$1$ surgery).
  This yields, after smoothing, an
  oriented closed manifold~$S'$ that agrees outside the chosen tubular
  neighbourhood of $Y$ with $S$.  We can do this construction also
  avoiding $Y'$ if necessary.  Note that
  $S' \cap Y = (S \cap Y) \setminus\{x_-, x_+\}$, and that $S$ and
  $S'$ are homologous, because
  $[S'] - [S] = \bigl[\partial (D^k \times [0,1])\bigr]$.

  By repeating this construction as often as necessary, we can cancel
  all pairs of intersection points of opposite sign until all points
  in $S' \cap Y$ have the same sign.  This then implies as desired
  that $\#\bigl(S \cap Y\bigr) = \bigl|[S]\cdot [Y]\bigr|$.
\end{proof}

\begin{lemma}\label{intersection line nodal curve is not trivial}
  Let $\ell$ be a surface in $\overline{W}$ that is transverse to the
  evaluation map~$\eval$.  Denote the oriented
  submanifold~$\eval^{-1}(\ell)$ in $\overline{\mM}_z(p_0, H_\infty)$
  by $\ell'$.

  If the intersection product
  $[\ell'] \cdot [{\mathcal N}_\infty^{(i)}]$ is trivial for all
  $i = 1, \dotsc, N$, then it follows that $\ell$ is null-homologous.
\end{lemma}
\begin{proof}
  Generically, $\ell$ is disjoint from $H_\infty$, so we may assume
  that $\ell'$ does not intersect $\eval^{-1}(H_\infty)$, and after a
  further perturbation we can assume that $\ell'$ is transverse to
  ${\mathcal N}_\infty^{(i)}$ without changing the homology class of
  $\ell'$.

  If $[\ell'] \cdot [{\mathcal N}_\infty^{(i)}] = 0$, we can apply
  Lemma~\ref{remove unnecessary intersection points} to find a
  surface~$\ell''$ in $\overline{\mM}_z(p_0, H_\infty)$ that is
  homologous to $\ell'$ and that does not have any intersection points
  either with ${\mathcal N}_\infty^{(i)}$ or with
  $\eval^{-1}(H_\infty)$.  Furthermore, since this modification has
  been performed in an arbitrarily small neighbourhood of
  ${\mathcal N}_\infty^{(i)}$, we may assume that we have not created
  any new intersection points with one of the other components
  ${\mathcal N}_\infty^{(j)}$ for $j\ne i$.

  Thus, if the intersection product
  $[\ell'] \cdot [{\mathcal N}_\infty^{(i)}]$ is trivial for all
  $i = 1, \dotsc, N$, we obtain by successively applying this
  construction for each $i$ a surface $\ell''$ in
  $\overline{\mM}_z(p_0, H_\infty)$ with $[\ell''] = [\ell']$ that
  does not intersect any of the ${\mathcal N}_\infty^{(i)}$ or
  $\eval^{-1}(H_\infty)$.

  We then have that $\eval_*([\ell''])=0$ as in the proof of
  Lemma~\ref{noncompact}, because
  $\overline{\mM}_z(p_0, H_\infty) \setminus \bigl( \bigcup_i
  {\mathcal N}_\infty^{(i)} \cup \eval^{-1}(H_\infty)\bigr)$ retracts
  to $\eval^{-1}(p_0)$, but due to Lemma~\ref{why we like it} we see
  that $\eval_*([\ell']) = [\ell]$.  Since $[\ell''] = [\ell']$, it
  follows that $[\ell] = 0$.
\end{proof}

\begin{lemma}\label{degree and intersection}
  Let $\ell\subset \overline{W}$ be a surface that is transverse to
  the evaluation map~$\eval$ and that represents the homology class of
  a line in $W_\infty$.  Then it follows for
  $\ell' = \eval^{-1}(\ell)$ that
  $$  [\ell'] \cdot \bigl[{\mathcal N}_\infty^{(i)}\bigr]
  = \deg(\eval_\infty^{(i)}) \;. $$
\end{lemma}
\begin{proof}
  Let $y \in H_\infty$ be a regular value of $\eval_\infty^{(i)}$ for
  all $i=1, \dotsc, N$, and let $\ell_0$ be a line in $W_\infty$
  intersecting $H_\infty$ transversely at $y$.  It follows that
  $\ell_0$ is transverse to $\eval|_{\mathcal{N}_\infty^{(i)}}$ at
  $y$, that is, for every
  $x \in \eval^{-1}(y) \cap {\mathcal N}_\infty^{(i)}$ we have
  \begin{equation}\label{transversality with N}
    T_y \ell_0 \oplus d_x\eval\bigl(T_x{\mathcal N}_\infty^{(i)}\bigr)
    = T_y \overline{W} \;,
  \end{equation}
  because the nodal $J$-holomorphic spheres in
  $\overline\mM(p_0, H_\infty)$ are all transverse to
  $W_\infty$.

  By construction
  $\bigl(\eval_\infty^{(i)}\bigr)^{-1}(y)= \bigl(\eval|_{{\mathcal
    N}_\infty^{(i)}}\bigr)^{-1} (\ell_0)$.  If
  $x \in \bigl(\eval|_{{\mathcal N}_\infty^{(i)}}\bigr)^{-1}(\ell_0)$, we define
  $\sign(x) = +1$ if the equality of Equation~\eqref{transversality
    with N} preserves the orientation, and $\sign(x)=-1$
  otherwise. Then $\sign(x)= \sign(d_x \eval_\infty^{(i)})$ because
  $d_x \eval$ is complex linear in the extra direction
  $T_x{\mathcal N}_\infty^{(i)}/T_x{\mathcal S}_\infty^{(i)}$.

  Now let $\ell$ be a small perturbation of $\ell_0$ which is
  transverse to $\eval$.  By Equation \eqref{transversality with N} we can assume 
that the perturbation is supported away from $y$ and that no new intersection points 
between $\eval^{-1}(\ell)$ and ${\mathcal N}_\infty^{(i)}$ are created. Then
  \begin{equation*}
      [{\mathcal N}_\infty^{(i)}] \cdot [\eval^{-1}(\ell)]
      =  \!\!\!\!\! \sum_{x \in \bigl(\eval|_{{\mathcal N}_\infty^{(i)}}\bigr)^{-1}(\ell_0)} \!\!\!\!\!\!\!\!\!\! \sign(x)
      = \!\!\!\!\! \sum \limits_{x \in \bigl(\eval^{(i)}_\infty \bigr )^{-1}(y)} \!\!\!\!\! \!\!\!\!\! \sign (d_x \eval^{(i)}_\infty) = 
\deg(\eval_\infty^{(i)}) \;.  \qedhere
  \end{equation*}
\end{proof}

\begin{lemma}\label{ultimissimo}
  There exists an $i \in \{1, \dotsc, N \}$ such that
  $\deg(\eval_\infty^{(i)}) \ne 0$.
\end{lemma}
\begin{proof}
  Let $\ell$ be a surface in $\overline{W}$ that is transverse to the
  evaluation map~$\eval$ and that represents the homology class of a
  line in $W_\infty$.  Denote $\eval^{-1}(\ell)$ by
  $\ell'$.  By Lemma~\ref{degree and intersection},
  $\deg(\eval_\infty^{(i)}) = [\ell'] \cdot \bigl[{\mathcal
    N}_\infty^{(i)}\bigr]$.  Thus, if $\deg(\eval_\infty^{(i)})$ were
  $0$ for every $i=1, \dotsc, N$, it would follow from
  Lemma~\ref{intersection line nodal curve is not trivial} that
  $[\ell]=0$.
  But this is impossible because the symplectic form
  evaluates positively on $[\ell]$.
\end{proof}

After all this preparation, the proof of Theorem~\ref{nonfillability} is falling at our feet like a
ripe fruit.

\begin{proof}[Proof of Theorem~\ref{nonfillability}]
  By Lemma~\ref{ultimissimo} there exists an $i \in \{1, \dotsc, N\}$
  such that $\deg(\eval_\infty^{(i)}) \ne 0$.  Then, by
  Lemma~\ref{ultimo},
  $[\ell]= 2\, \eval_*\bigl([A_\infty^{(i)}]\bigr)$ modulo torsion.
  If we evaluate the first Chern class of $T\overline{W}$ on $[\ell]$
  we obtain
  $$
  n+2 = \bigl\langle c_1(T \overline{W}), [\ell] \bigr\rangle
  = 2\, \bigl\langle c_1(T \overline{W}),
  \eval_*\bigl([A_\infty^{(i)}]\bigr)\bigr\rangle \;,
  $$
  which is a contradiction when $n$ is odd.
\end{proof}

\section{Fundamental group of semipositive fillings}\label{sec: pi_1}

In this section let $(W, \omega)$ be a semipositive filling of
$(\R\PP^{2n-1}, \xi)$.  We recall that $(W, \omega)$ is semipositive
if every class~$A$ in the image of the Hurewicz homomorphism
$\pi_2(W) \to H_2(W; \Z)$ satisfying the conditions
$\langle \omega, A \rangle >0$ and
$\langle c_1(TW), A \rangle \ge 3-n$ also satisfies
$\langle c_1(TW), A \rangle \ge 0$.  See \cite[Definition 6.4.1]{MS}.

We use the same compactification~$(\overline{W}, \overline{\omega})$
and the same set of almost complex structures ${\mathcal J}$ as in the
previous sections, but now that $(W, \omega)$ does not need to be
symplectically aspherical we cannot assume anymore that
$\overline\mM(p_0, H_\infty)$ is a manifold or that its elements have
no irreducible component contained completely inside
$\overline{W} \setminus W_\infty$.  The irreducible components which
intersect $W_\infty$ must be simply covered because the intersections
are simple, and therefore are Fredholm regular for a generic almost
complex structure $J \in {\mathcal J}$, but the irreducible components
which are contained in $\overline{W} \setminus W_\infty$ can be
multiply covered.  However according to \cite[Theorem 6.6.1]{MS} the
image of $\mM^{\op{red}}(p_0, H_\infty)$ under the evaluation map is
contained in the union of images of finitely many compact codimension
two smooth manifolds for a generic $J \in {\mathcal J}$ because the
irreducible components intersecting $W_\infty$ are Fredholm regular
and the irreducible components contained in
$\overline{W} \setminus W_\infty$ are controlled by semipositivity.
In particular
$\overline{W} \setminus \eval(\mM^{\op{red}}(p_0, H_\infty))$ is open,
dense and connected. Moreover the restriction of the evaluation map
\begin{equation*}
  \eval \colon
  \overline{\mathcal M}(p_0, H_\infty) \setminus
  \eval^{-1}\bigl(\eval(\mM^{\op{red}}(p_0, H_\infty))\bigr)
  \to \overline{W} \setminus \eval\bigl(\mM^{\op{red}}(p_0, H_\infty)\bigr)
\end{equation*}
is proper by Gromov compactness, and therefore its degree is well
defined.  Then Lemma~\ref{evaluation has degree one} can be rephrased
as follows.

\begin{lemma}
  If $(W, \omega)$ is semipositive and
  $y \in \overline{W} \setminus \eval(\mM^{\op{red}}(p_0, H_\infty))$
  is a regular value of $\eval$, then
  $$\sum \limits_{x \in \eval^{-1}(y)} \sign(d_x\eval)=1  .$$
  In particular,
  $\eval \colon \overline\mM_z(p_0, H_\infty) \to \overline{W}$ is
  surjective.
\end{lemma}

If we apply the argument of Lemma~\ref{noncompact} to a
$1$-dimensional submanifold of $W$ we obtain the following result.

\begin{lemma}\label{pi1 surjective}
  If $(W, \omega)$ is a semipositive symplectic filling of
  $(\R\PP^{2n-1}, \xi)$, then the inclusion
  $\iota \colon \R\PP^{2n-1} \to W$ induces a surjective map
  $\iota_*\colon \pi_1\bigl(\R\PP^{2n-1}\bigr) \to \pi_1(W)$.
\end{lemma}
\begin{proof}
  Recall that $\overline{W}\setminus W_\infty$ is equal to
  $W\setminus \partial W$.  Instead of proving that
  $\pi_1(\partial W)$ maps surjectively onto $\pi_1(W)$, we can
  equivalently show that for a sufficiently small
  neighbourhood~$U_\epsilon$ of $W_\infty$,
  $\pi_1(U_\epsilon\setminus W_\infty)$ is surjective in
  $\pi_1(\overline{W}\setminus W_\infty)$.

  Choose a base point~$b$ for $\pi_1(\overline{W}\setminus W_\infty)$
  that lies in the neighbourhood~$U$ of Lemma~\ref{evaluation has
    degree one}, and use $b' = \eval^{-1}(b)$ as the base point for
  $\pi_1(\mM_z(p_0, H_\infty))$.
  
  We can represent every element of
  $\pi_1\bigl(\overline{W}\setminus W_\infty\bigr)$ by a smooth
  embedding
  \begin{equation*}
    \gamma \colon S^1 \hookrightarrow W
  \end{equation*}
  that avoids $\eval\bigl(\mM^{\op{red}}(p_0, H_\infty)\bigr)$ by a
  codimension argument and that is transverse to the evaluation map.
  Using the fact that $\eval$ is a diffeomorphism of $U$ onto its
  image and arguing as in point~(i) of the proof of \cite[Lemma
  2.3]{gnw} we obtain a loop
  $\Gamma \colon S^1 \to \mM_z(p_0, H_\infty)$ such that
  $\Gamma(1)= b'$ and $\eval_*([\Gamma]) = [\gamma]$ in
  $\pi_1(\overline{W}\setminus W_\infty)$.  Furthermore $\Gamma$ does
  not intersect any singular stratum or $\eval^{-1}(W_\infty)$.

  We can isotope $\mM_z(p_0, H_\infty) \setminus \eval^{-1}(H_\infty)$
  into an arbitrarily small neighbourhood of $\eval^{-1}(p_0)$ by
  pushing the marked point in every holomorphic sphere from $\infty$
  towards $0$.  This isotopy restricts to
  $\mM_z(p_0, H_\infty) \setminus \eval^{-1}(W_\infty)$, so that
  $\Gamma$ is homotopic in
  $\mM_z(p_0, H_\infty) \setminus \eval^{-1}(W_\infty)$ to a loop in a
  neighbourhood of $\eval^{-1}(p_0)$.

  Then it follows that $[\gamma]$ is homotopic in
  $\overline{W}\setminus W_\infty$ to a loop that lies in an
  arbitrarily small neighbourhood of $p_0$ and
  $\pi_1(U_\epsilon\setminus W_\infty) \to
  \pi_1\bigl(\overline{W}\setminus W_\infty\bigr)$ is surjective.
\end{proof}

Combining this with the argument found in \cite[Section~6.2]{ekp} we
obtain the main result of this section.

\begin{thm}\label{simply connected}
  Any semipositive symplectic filling of $(\R\PP^{2n-1}, \xi)$ is
  simply connected.
\end{thm}
\begin{proof}
  Let $(W, \omega)$ be a semipositive symplectic filling of
  $(\R\PP^{2n-1}, \xi)$.  By Lemma~\ref{pi1 surjective} the map
  $\pi_1(\partial W) \to \pi_1(W)$ induced by the inclusion
  $\iota\colon \partial W \hookrightarrow W$ is surjective so that
  $\pi_1(W)$ is either trivial or isomorphic to $\Z/2\Z$.

  In the latter case $\iota$ induces an isomorphism between the
  fundamental groups, and thus
  $$\iota^* \colon H^1(W; \Z / 2 \Z) \to H^1\bigl(\partial W; \Z / 2 \Z\bigr)$$
  is also an isomorphism.  Let $\alpha \in H^1(W; \Z / 2 \Z)$ be the
  nontrivial element.  Then
  $\iota^* \alpha \in H^1\bigl(\partial W; \Z / 2 \Z\bigr)$ is also
  nontrivial and, since $H^*(\R \PP^{2n-1}; \Z / 2 \Z)$ is generated
  as an algebra by the nontrivial element of degree one,
  $(\iota^*\alpha)^{2n-1}$ is the nontrivial element of
  $H^{2n-1}\bigl(\partial W; \Z / 2 \Z\bigr)$.

  By the naturality of the cup product
  $(\iota^*\alpha)^{2n-1}= \iota^*(\alpha^{2n-1})$.  However
  \begin{equation*}
    \iota_*\colon H_{2n-1}\bigl(\partial W; \Z / 2 \Z\bigr) \to
    H_{2n-1}(W; \Z / 2 \Z)
  \end{equation*}
  is trivial, and consequently
  $\iota^*\colon H^{2n-1}(W; \Z / 2 \Z) \to H^{2n-1}\bigl(\partial W;
  \Z / 2 \Z\bigr)$ is also trivial by duality because we are working
  over a field.  This contradicts $\iota^*(\alpha^{2n-1}) \ne 0$ and
  therefore shows that $W$ is simply connected.
\end{proof}

\section{Yet another proof of the Eliashberg-Floer-McDuff 
theorem}

In this section we apply the constructions of this article to the
symplectic fillings of the standard contact structure $\xi$ on
$S^{2n-1}$.  This will lead to small changes in the meaning of the
notation.  If $(W, \omega)$ is a symplectic filling of
$(S^{2n-1}, \xi)$ and we perform symplectic reduction of its boundary,
we obtain a closed symplectic manifold
$(\overline{W}, \overline{\omega})$ with a codimension two symplectic
submanifold $ W_\infty \cong \C\PP^{n-1}$ whose normal bundle is
isomorphic to $\oO_{\PP^{n-1}}(1)$.  We choose an almost complex
structure $J$ on $\overline{W}$ which is integrable near $W_\infty$
and generic elsewhere. Let $p_0 \in W_\infty$ be a point; we denote by
$\mM(p_0)$ the moduli space of unparametrised $J$-holomorphic spheres
in $\overline{W}$ that are homotopic to a line in $W_\infty$ and pass
through $p_0$.  If $\ell$ is a line in $W_\infty$, then
$$T\overline{W}|_\ell \cong \oO_{\PP^1}(2) \oplus 
\underbrace{\oO_{\PP^1}(1) \oplus \dotsm \oplus \oO_{\PP^1}(1)}_{n-1}
.$$ Since $[\ell] \cdot [W_\infty]=1$ all elements of $\mM(p_0)$ are
simply covered, and therefore $\mM(p_0)$ is a smooth manifold by the
analogue of Proposition~\ref{genericity}.  Let $\mM_z(p_0)$ is the
moduli space obtained by adding a free marked point to the elements of
$\mM(p_0)$.  A Riemann-Roch calculation gives $\dim \mM(p_0) = 2n-2$
and $\dim \mM_z(p_0) = 2n$.

\begin{lemma}
  If $(W, \omega)$ is symplectically aspherical, then $\mM_z(p_0)$ is
  compact.
\end{lemma}
\begin{proof}
  As the algebraic intersection between a line with $W_\infty$ is one,
  any nodal $J$-holomorphic curve representing the homology class of a
  line must have an irreducible component in
  $\overline{W} \setminus W_\infty \cong W$.
\end{proof}

Lemma~\ref{evaluation has degree one} still holds with the minimal
necessary modifications, and therefore the evaluation map
$\eval \colon \mM_z(p_0) \to \overline{W}$ has degree one.

\begin{lemma}
  If $(W, \omega)$ is a symplectically aspherical filling of
  $(S^{2n-1}, \xi)$, then $H_*(W; \Z)=0$ for $*>0$.
\end{lemma}
\begin{proof}
  The moduli space $\mM_z(p_0)$ is an $S^2$-bundle over $\mM(p_0)$ and
  $\eval^{-1}(p_0)$ is a section.  Let $\widetilde{W}_\infty$ be a
  $J$-holomorphic hypersurface of $\overline{W}$ contained in the
  neighbourhood of $W_\infty$ where $J$ is integrable and obtained as
  the graph of a section of $\oO_{\PP^{n-1}}(1)$.  We choose
  $\widetilde{W}_\infty$ such that
  $p_0 \not \in \widetilde{W}_\infty$: then
  $\eval^{-1}(\widetilde{W}_\infty)$ is a section of $\mM_z(p_0)$
  which is disjoint from $\eval^{-1}(p_0)$.  The map
  \begin{equation}\label{fastidio}
    \eval_* \colon H_*\bigl(\mM_z(p_0)
    \setminus \eval^{-1}(\widetilde{W}_\infty); \Z\bigr) \to
    H_*\bigl(\overline{W} \setminus \widetilde{W}_\infty; \Z\bigr) \cong H_*(W;
    \Z)
  \end{equation}
  is surjective by Lemma~\ref{why we like it}. 
  That lemma, strictly
  speaking, is about homology classes represented by submanifolds, but
  there are several ways to extend it to general homology classes.

  On the other hand
  $\mM_z(p_0) \setminus \eval^{-1}(\widetilde{W}_\infty)$ retracts
  onto $\eval^{-1}(p_0)$, and therefore the map~\eqref{fastidio} is
  trivial for $*>0$.
\end{proof}

The proof of Lemma~\ref{pi1 surjective} works with the obvious
modifications more or less unchanged for fillings of
$(S^{2n-1}, \xi)$, and therefore $W$ is simply connected.  Then the
$h$-cobordism theorem implies the following corollary.

\begin{cor}[Eliashberg--Floer--McDuff]
  If $(W, \omega)$ is a symplectically aspherical filling of
  $(S^{2n-1}, \xi)$, then $W$ is diffeomorphic to the ball $D^{2n}$.
\end{cor}

\bibliographystyle{amsplain}

\providecommand{\bysame}{\leavevmode\hbox to3em{\hrulefill}\thinspace}
\providecommand{\MR}{\relax\ifhmode\unskip\space\fi MR }
\providecommand{\MRhref}[2]{%
  \href{http://www.ams.org/mathscinet-getitem?mr=#1}{#2}
}
\providecommand{\href}[2]{#2}

\end{document}